\documentclass[a4paper,12pt]{amsart}

\usepackage{amsmath} 
\usepackage{amssymb}
\usepackage[mathscr]{eucal}
\usepackage{graphics}
\usepackage{multicol}
\usepackage{braket}
\usepackage{longtable}

\usepackage{newtxtext, newtxmath}

\usepackage{ascmac}
\usepackage{framed}
\usepackage{ulem}
\usepackage{cancel}
\usepackage{arydshln}
\usepackage{bm}
\allowdisplaybreaks
\theoremstyle{plain}
\newtheorem{thm}{Theorem}[section]
\newtheorem{lem}[thm]{Lemma}
\newtheorem{prop}[thm]{Proposition}

\makeatletter
\def\thefootnote{\ifnum\c@footnote>\z@\leavevmode\lower.5ex%
      \hbox{$^{\@arabic\c@footnote)}$}\fi}
\makeatother

\makeatletter
\def\section{\@startsection{section}{1}%
    \z@{.7\linespacing\@plus\linespacing}{.5\linespacing}%
    {\normalfont\sc\bfseries\centering}}

\def\subsection{\@startsection{subsection}{2}%
    \z@{.5\linespacing\@plus.7\linespacing}{.3\linespacing}%
    {\normalfont\sc\bfseries}}
\makeatother

\theoremstyle{definition}

\theoremstyle{remark}

\def\dfrac#1#2{\displaystyle \frac{#1}{#2}}
\def\Aut{\mbox{\rm {Aut}}}
\def\det{\mbox{\rm {det}}}
\def\diag{\mbox{\rm {diag}}}

\def\Ad{\mbox{\rm {Ad}}}
\def\ad{\mbox{\rm {ad}}}

\def\Iso{\mbox{\rm {Iso}}}

\def\Ker{\mbox{\rm {Ker}}}

\def\tr{\mbox{\rm {tr}}}

\def\ti{\tilde}

\def\dsum{\displaystyle \sum}

\def\dfrac#1#2{\displaystyle \frac{#1}{#2}}

\def\R{\mbox{\boldmath $R$}}

\def\Z{\mbox{\boldmath $Z$}}

\def\0{\mbox{\boldmath {0}}}    
\def\1{\mbox{\boldmath {1}}}      
\def\2{\mbox{\boldmath {2}}}      
\def\3{\mbox{\boldmath {3}}}      
\def\4{\mbox{\boldmath {4}}}      
\def\5{\mbox{\boldmath {5}}}      
\def\6{\mbox{\boldmath {6}}}      
\def\7{\mbox{\boldmath {7}}}      
\def\8{\mbox{\boldmath {8}}}      
\def\9{\mbox{\boldmath {9}}}      
\def\a{\mbox{\boldmath $a$}}

\def\dd{\mbox{\boldmath $d$}}
\def\e{\mbox{\boldmath $e$}}

\def\u{\mbox{\boldmath $u$}}
\def\v{\mbox{\boldmath $v$}}
\def\w{\mbox{\boldmath $w$}}
\def\x{\mbox{\boldmath $x$}}
\def\y{\mbox{\boldmath $y$}}

\usepackage{epic,eepic}

\begin{document}
	
\title[Realizations of inner automorphisms of order four on $E_8$ 
Part II]
{Realizations of inner automorphisms of order four and fixed points subgroups by them on the connected compact exceptional Lie group $E_8$,  Part II}

\author[Toshikazu Miyashita]{By \vspace{3mm} \\ \vspace{3mm} Toshikazu Miyashita}



\address{1365-3 Bessho onsen    \endgraf
	Ueda City                     \endgraf
	Nagano Prefecture 386-1431    \endgraf
	Japan}
\email{anarchybin@gmail.com}

\begin{abstract}
The compact simply connected Riemannian 4-symmetric spaces were 
classified by J.A. Jim{\'{e}}nez according to type of the Lie algebras.
As homogeneous manifolds, these spaces are of the 
form $G/H$, where $G$ is a connected compact simple Lie group with 
an automorphism $\tilde{\gamma}$ of order four on $G$ and  $H$ is a 
fixed points subgroup $G^\gamma$ of $G$. According to the 
classification by J.A. Jim{\'{e}}nez, there exist seven compact 
simply connected Riemannian 4-symmetric spaces $ G/H $ in the case where $ G $ is of type $ E_8 $. In the present article, 
we give the explicit form of automorphisms  $\tilde{w}_{{}_4}
\tilde{\upsilon}_{{}_4}$ and $\tilde{\mu}_{{}_4}$ of order four on 
$E_8$ induced by the $C$-linear transformations $w_{{}_4}, 
\upsilon_{{}_4}$ and $\mu_{{}_4}$ of the 248-dimensional vector 
space ${\mathfrak{e}_8}^{C}$, respectively.  
Further, we determine the structure of these fixed points subgroups 
$(E_8)^{w_{{}_4}}, (E_8)^{{}_{\upsilon_{{}_4}}}$ and $(E_8)^{{}
_{\mu_{{}_4}}}$ of $ E_8 $.
These amount to the global realizations of three spaces 
among seven Riemannian 4-symmetric spaces $ G/H $ above 
corresponding to the Lie algebras $ \mathfrak{h}=i\bm{R} \oplus 
\mathfrak{su}(8), i\bm{R} \oplus \mathfrak{e}_7$ and $\mathfrak{h}=
\mathfrak{su}(2) \oplus \mathfrak{su}(8)$, where $ \mathfrak{h}={\rm 
Lie}(H) $. 
\end{abstract}

\maketitle

\section {Introduction}
Let $G$ be a Lie group and $H$ a compact subgroup of $G$. A 
homogeneous space $G/H$ with $G$-invariant Riemannian metric $g$ is 
called a Riemannian $4$-symmetric space if there exists an 
automorphism $\tilde{\gamma}$ of order four on $G$ such that $({G^
\gamma})_{0} \subset H \subset G^\gamma$, where $G^\gamma$ and $({G^
\gamma})_{0}$ are the fixed points subgroup of $G$ by $
\tilde{\gamma}$ and its identity component, respectively.

Now, for the exceptional compact Lie group of type $E_8$, as in Table  
below, there exist seven cases of the compact simply connected 
Riemannian $4$-symmetric spaces which were classified by J.A. 
Jim\'{e}nez as mentioned in abstract (\cite{Jim}). Accordingly, our 
interest is to realize the groupfication for the classification as 
Lie algebra. 

Our results of groupfication corresponding to the Lie algebra $
\mathfrak{h}$ in Table are given as follows.

		\begin{longtable}[c]{clll}
			\noalign{\hrule height 1pt}
			\noalign{\vspace{-1pt}} 
			&&&\\[-6pt]
			Case & \hspace*{7mm}  $\mathfrak{h}$  &  $\tilde{\gamma}
			$  &  \hspace*{15mm}  $H=G^\gamma$ 
            \\
            \noalign{\vspace{3pt}}
            \noalign{\hrule height 0.5pt}
			&&&\\[-8pt]
			1 
			& $\mathfrak{so}(6) \oplus \mathfrak{so}(10)$ 
			& $\tilde{\sigma}'_{{}_4}$ 
			& $(Spin(6) \times Spin(10))/\bm{Z}_4$
			\\[6pt]
			2 
			& $i\bm{R} \oplus \mathfrak{su}(8)$  
			& $\tilde{w}_{{}_4}$ 
			& $(U(1) \times SU(8))/\bm{Z}_{24}$
			\\[6pt]		
			3 
			& $i\bm{R} \oplus \mathfrak{e}_7 $ 
			& $\tilde{\upsilon}_{{}_4}$ 
			& $(U(1) \times E_7)/\bm{Z}_2$
			\\[6pt]
			4 
			& $\mathfrak{su}(2) \oplus \mathfrak{su}(8)$ 
			& $\tilde{\mu}_{{}_4}$ 
			& $(SU(2) \times SU(8))/\bm{Z}_4$
			\\[6pt]
			5 
			& $\mathfrak{su}(2) \oplus i\bm{R} \oplus \mathfrak{e}_6$ 
			& $\tilde{\omega}_{{}_4}$ 
			& $(SU(2) \times U(1) \times E_6)/(\bm{Z}_2 \times \bm{Z}_3)$
			\\[6pt]
			6 
			& $i\bm{R} \oplus \mathfrak{so}(14)$ 
			& $\tilde{\kappa}_{{}_4}$ 
			& $(U(1) \times Spin(14))/\bm{Z}_4$
			\\[6pt]
			7 
			& $\mathfrak{su}(2) \oplus i\bm{R} \oplus \mathfrak{so}(12)$ 
			& $\tilde{\varepsilon}_{{}_4}$ 
			& $(SU(2) \times  U(1) \times Spin(12))/(\bm{Z}_2 \times \bm{Z}_2)$
			\\[6pt]
			\noalign{\hrule height 1pt}
		\end{longtable}


In \cite{miya2}, the author showed the group realizations for Case 
$ 1 $ in Table. In the present article, we state the 
realizations of the group $ H $ for Case $ 2,3 $ and $ 4 $. The 
remaining cases will be shown in a forthcoming article \cite{miya3} 
by the author.

Finally, the author would like to say that the feature of this 
article is to give elementary proofs of the isomorphism of groups 
by using the homomorphism theorem except several proofs, and of the 
connectedness of groups as topological spaces. 

This article is a continuation of \cite{miya2}, hence we start from 
Section 4. We refer the reader to \cite{miya2} for preliminary 
results and also to \cite{miya1}, \cite{miya2}, \cite{iy3}, \cite{iy1} or 
\cite{iy0} for notations.
Note that we change the numbering of Case 5 and Case 6 in 
\cite{miya2} to the numbering of Case 3 and Case 4 in the present 
article, respectively.


\setcounter{section}{3}

\section { Case 2. The automorphism $\tilde{w}_{{}_4}$ of order four 
and the group $(E_8)^{w_{{}_4}}$}

In this case, we will study the connected compact exceptional Lie group 
of type 
$ E_8 $ constructed by S. Gomyo (\cite{go}). With reference to 
\cite{go}, we rewrite its contents as detailed as possible. In 
particular, we give some proofs of lemma and theorem in which the proofs 
are omitted in \cite{go}.

First, in order to construct another $C$-Lie algebra of 
type $E_8$, we investigate the properties of the exterior $
C$-vector space ${\varLambda}^3(C^9)$. Let $\e_1, \cdots, \e_9$ be the 
canonical $C$-basis of the nine dimensional $C$-vector 
space $
C^9$ and $
(\x, \y)$ the inner product in $C^9$ satisfying $(\e_i, 
\e_j) = 
\delta_{ij}$, where $\delta_{ij}$ means the Kronecker's delta. In $
{\varLambda}^3(C^9)$, we define an inner product by
\begin{align*}
	\begin{array}{c}
    (\u_1 \wedge \u_2 \wedge \u_3, \v_1 \wedge \v_2 \wedge \v_3) = \det\Big((\u_i, \v_j)\Big), 
	\vspace{1mm}\\
    (a, b) = ab, \,\, a, b \in {\varLambda}^0(C^9) = 
    C.
	\end{array}
\end{align*}
Here, $\e_{i_1} \wedge \e_{i_2} \wedge \e_{i_3}, i_1 < i_2 < i_3$ 
forms an orthonormal $C$-basis of ${\varLambda}^3(C^9)$.
For $\u \in {\varLambda}^3(C^9)$, we define an element $
\ast \u \in {\varLambda}^6(C^9)$ by
\begin{align*}
	(*\u, \v) = (\u \wedge \v, \e_1 \wedge \cdots \wedge \e_9), \,\, 
	\v \in {\varLambda}^6(C^9). 
\end{align*}
Note that the inner product $(*\u, \v)$ is defined as in the case $
{\varLambda}^3(C^9)$ above. Then, $*$ induces a $C
$-linear isomorphism
\begin{align*}
	* : {\varLambda}^3(C^9) \to {\varLambda}^6 (C
	^9) 
\end{align*}
and satisfies the following identity 
\begin{align*}
	   *^2\u = \u, \,\, \u \in {\varLambda}^3(C^6). 
\end{align*}

The Lie algebra $\mathfrak{sl}(9,C)$ of the group $SL(9,
C)$ acts on ${\varLambda}^3(C^9)$ as follows:
\begin{align*}
	&D(\u_1 \wedge \u_2 \wedge \u_3) =D\u_1 \wedge \u_2 \wedge \u_3
                                +\u_1 \wedge D\u_2 \wedge \u_3 
                                +\u_1 \wedge \u_2 \wedge D\u_3,
	\\
	&D(1) = 0,\,\, D \in \mathfrak{sl}(9,C). 
\end{align*}

\begin{lem}[{\cite[Lemma 1.1]{go}}]\label{lemma 4.1}
	For $D \in \mathfrak{sl}(9, C)$ and $\u, \v \in 
	{\varLambda}^3(C^9)$, we have the following relational 
	formulas
	\begin{align*}
		(D\u, \v) = (\u, {}^tD\v), \,\, *(D\u) = -{}^tD(*\u).
	\end{align*}
	 
%

\end{lem}
\begin{proof}
 In order to prove this lemma, it is sufficient to prove these above 
 for $\u=\u_1 \wedge \u_2 \wedge \u_3, \v=\v_1 \wedge \v_2 \wedge 
 \v_3 \in \varLambda^3 (C^9)$. 
For the former formula, it follows that
  \begin{align*}
	(D\u,\v)=\det\Bigl( (D\u_i, \v_j)\Bigr)=\det\Bigl( (\u_i, {}^t
	D\v_j)\Bigr)=(\u,{}^tD\v).
 \end{align*}
 For the latter formula,  
 using the relational formula $D(\u \wedge \w)=D\u 
 \wedge \w +\u \wedge D\w$, we have the following
 \begin{align*}
		(*(D\u), \w)
		&=(D\u \wedge \w, \e_1 \wedge \cdots \wedge \e_9)
		\\
		&=(D(\u \wedge \w)-\u \wedge D\w, \e_1 \wedge \cdots \wedge \e_9)
		\\
		&=(D(\u \wedge \w), \e_1 \wedge \cdots \wedge \e_9)
		-(\u \wedge D\w, \e_1 \wedge \cdots \wedge \e_9)
		\\
		&=(\u \wedge \w, {}^tD(\e_1 \wedge \cdots \wedge \e_9))
		-(\u \wedge D\w, \e_1 \wedge \cdots \wedge \e_9)
		\\
		&=(\u \wedge \w, \sum_{i=1}^{9}\e_1 \wedge \cdots \wedge {}^tD\e_i \wedge \cdots \wedge \e_9)
		-(\u \wedge D\w, \e_1 \wedge \cdots \wedge \e_9)
		\\
		&=(\u \wedge \w, \tr({}^tD)(\e_1 \wedge \cdots \wedge \e_9))
		-(\u \wedge D\w, \e_1 \wedge \cdots \wedge \e_9)
		\\
		&=-(\u \wedge D\w, \e_1 \wedge \cdots \wedge \e_9)
		\\
		&=(-*\u,D \w)
		\\
		&=-({}^tD (*\u),\w)
 \end{align*}
 for every $\w \in {\varLambda}^6(C^9)$. Hence we have $*(D
 \u) = -{}^tD(*\u)$.

\end{proof}

For $\u, \v \in {\varLambda}^3(C^9)$, we define a $
C$-linear transformation $\u \times \v$ of $C^9$ 
by
\begin{align*}
(\u \times \v)\x = *(\v \wedge *(\u \wedge \x)) + \dfrac{2}{3}(\u, 
\v)\x, \,\, \x \in C^9.
\end{align*}

Then we have the following lemma.

\begin{lem}[{\cite[p.597]{go}}]\label{lemma 4.2}
	 We have $\tr \,(\u \times \v) = 0$.
\end{lem}
\begin{proof}
	In order to prove this lemma, we have to show the following relational formula
	\begin{align*}
		( *(\v \wedge *(\u \wedge \x)), \y)=-(\x \wedge \u, \y 
		\wedge \v),\,\, \x, \y \in C^9.
	\end{align*}
	Indeed, it follows from $*\u=\u, \u \in {\varLambda}
	^3(C^9)$ that
	\begin{align*}
		( *(\v \wedge *(\u \wedge \x)), \y)&=( (\v \wedge *(\u \wedge \x)) \wedge \y, \e_1 \wedge \cdots \wedge \e_9)
		\\
		&=-( *(\u \wedge \x) \wedge (\v \wedge \y), \e_1 \wedge \cdots \wedge \e_9)
		\\
		&=-(*^2 (\u \wedge \x), \v \wedge \y)
		\\
		&=-(\u \wedge \x, \v \wedge \y)
		\\
		&=-(\x \wedge \u, \y \wedge \v).
	\end{align*}
 Using $\tr \,(\u \times \v) =\displaystyle{\sum_{i=1}^9} ((\u 
 \times \v)\e_i, \e_i)$, we do straightforward computation of $\tr 
 \,
 (\u \times \v)$: 
	\begin{align*}
		\tr \,(\u \times \v)&=\sum_{i=1}^9 ((\u \times \v)\e_i, \e_i)
		=\sum_{i=1}^9 (*(\v \wedge *(\u \wedge \e_i))+\,\dfrac{2}{3}(\u, \v)\e_i, \e_i)
		\\[0mm]
		&=\sum_{i=1}^9 \left( (*(\v \wedge *(\u \wedge \e_i)), \e_i)+\dfrac{2}{3}(\u, \v)(\e_i, \e_i)\right)
		\\[0mm]
		&=\sum_{i=1}^9 ( (*(\v \wedge *(\u \wedge \e_i)), \e_i)+6(\u, \v)
		\\[0mm]
		&=-\sum_{i=1}^9 (\e_i \wedge \u, \e_i \wedge \v)+6(\u, \v)\,\,
		\biggl( \begin{array}{l}\u=\u_1 \wedge \u_2 \wedge \u_3 \\ 
		\v=\v_1 \wedge \v_2 \wedge \v_3
		\end{array}\biggr)
		\\[0mm]
		&=-\sum_{i=1}^9 \det\,
		\begin{pmatrix}
		(\e_i, \e_i) & (\e_i, \v_1) & (\e_i, \v_2) & (\e_i, \v_3)  \\
    	(\u_1, \e_i) & (\u_1, \v_1) & (\u_1, \v_2) & (\u_1, \v_3) \\
    	(\u_2, \e_i) & (\u_2, \v_1) & (\u_2, \v_2) & (\u_2, \v_3) \\
    	(\u_3, \e_i) & (\u_3, \v_1) & (\u_3, \v_2)  & (\u_3, \v_3) \\
    	\end{pmatrix}+6(\u, \v)
		\\[0mm] 
 		&=-\sum_{i=1}^9 \Bigl((\e_1,\e_i)(\u,\v)-(\e_i,\v_1)(\e_i \wedge \v_2 \wedge \v_3)+(\e_i,\v_2)(\e_i \wedge \v_1 \wedge \v_3)
 		\\
 		&\hspace*{64mm}-(\e_i,\v_3)(\e_i \wedge \v_1 \wedge \v_2)\Bigr)+6(\u, \v)
		\\[0mm] 
		&=-\Bigl(9(\u,\v)-(\u,\v_1 \wedge \v_2 \wedge \v_3)
		+(\u,\v_2 \wedge \v_1 \wedge \v_3)
		-(\u,\v_3 \wedge \v_1 \wedge \v_2)\Bigr)
		\\
		&\hspace*{100mm}+6(\u,\v)
		\\
		&=-\Bigl(9(\u,\v)-(\u,\v_1 \wedge \v_2 \wedge \v_3)
		-(\u,\v_1 \wedge \v_2 \wedge \v_3)
		-(\u,\v_1 \wedge \v_2 \wedge \v_3)\Bigr)
		\\
		&\hspace*{100mm}+6(\u,\v)
 		\\
		&=-\Bigl( 9(\u,\v)-(\u,\v)-(\u,\v)-(\u,\v)\Bigr)+6(\u,\v)
		\\
		&=0.
	\end{align*}
\end{proof}

Note that from Lemma \ref{lemma 4.2}, $\u \times \v$ can be regarded 
as an element of $\mathfrak{sl}(9,C)$ with respect to the 
canonical basis of $C^9$. 

\begin{lem}[{\cite[Lemma 1.2]{go}}]\label{lemma 4.3}
	For $A \in SL(9, C), D \in \mathfrak{sl}(9, C)
	$ and $\u, \v \in {\varLambda}^3(C^9)$, we have the 
	following 
	relational formulas

	{\rm (1)} $\; A(\u \times \v)A^{-1} = A\u \times {}^tA^{-1}\v, \,\, [D, \u \times \v] = D\u \times \v + \u \times (-{}^tD\v)$,

	{\rm (2)} $\; {}^t(\u \times \v) = \v \times \u$,
	
	{\rm (3)} $\;\tau(\u \times \v) = \tau\u \times \tau\v$, where 
	$\tau$ is the complex conjugation on ${\varLambda}^3(C
	^9)$,

	{\rm (4)} $\;\tr(D(\u \times \v)) = (D\u, \v)$.
\end{lem}
\begin{proof}	
	(1) As for the left hand side, it follows from the definition of 
	$\u \times \v$ that
\begin{align*}
 	A(\u \times \v)A^{-1}\x&=A (*(\v \wedge *(\u \wedge A^{-1}\x))+ 
 	\dfrac{2}{3}(\u, \v)(A^{-1}\x)),\,\, \x \in C^9
 	\\
 	&=A (*(\v \wedge *(\u \wedge A^{-1}\x))+\dfrac{2}{3}(\u, \v)\x,
\end{align*}
	on the other hand, as for the right hand side, it follows from Lemma \ref{lemma 4.1}  that 
	\begin{align*}
	(A\u \times {}^t A^{-1})\x&=*({}^t A^{-1}\v \wedge *(A\u \wedge 
	\x))+\dfrac{2}{3}(A\u, {}^t A^{-1}\v)\x,\,\, \x \in C^9
	\\
	&=*({}^t A^{-1}\v \wedge *(A(\u \wedge A^{-1}\x)))+\dfrac{2}{3}
	(\u,\v)\x
	\\
	&=*({}^t A^{-1}\v \wedge {}^t A^{-1}(*(\u \wedge A^{-1}\x)))+
	\dfrac{2}{3}(\u,\v)\x
	\\
	&=*({}^t A^{-1}(\v \wedge *(\u \wedge A^{-1}\x)))+\dfrac{2}{3}
	(\u,\v)\x
	\\
	&= A(*(\v \wedge *(\u \wedge A^{-1}\x)))+\dfrac{2}{3}(\u,\v)\x.
\end{align*}
	Hence we have that $A(\u \times \v)A^{-1}\x=(A\u \times {}^t 
	A^{-1})\x, \x \in C^9$, that is, $A(\u \times \v)A^{-1} 
	= A
	\u \times {}^tA^{-1}\v$.
	The relational formula $[D, \u \times \v] = D\u \times \v + \u \times (-{}^tD\v)$ is shown as above.
\vspace{1mm}

	(2) Since $\u \times \v \in \mathfrak{sl}(9,C)$ (Lemma 
	\ref{lemma 4.2}), we have $({}^t (\u \times \v)\x, \y)=(\x, (\u 
	\times \v)\y), \x, \y \in C^9$. Subsequently, we have 
	to 
	show $(\x, (\u \times \v)\y)=((\v \times \u)\x, \y)$.
	Indeed, as for the 
	left 
	hand side, it follows that 
	\begin{align*}
 		(\x, (\u \times \v)\y)
 		&=(\x, *(\v \wedge *(\u \wedge \y))+\dfrac{2}{3}(\u, \v)\y) 
 		\\
		&=(\x, *(\v \wedge *(\u \wedge \y)))+\dfrac{2}{3}(\u, \v)(\x,\y) 
		\\
		&=(*(\v \wedge *(\u \wedge \y)), \x)+\dfrac{2}{3}(\u, \v)(\x,\y) 
		\\
		&=-(\y \wedge \u, \x \wedge \v)+\dfrac{2}{3}(\u, \v)(\x,\y),
	\end{align*}
	on the other hand, as for the right hand side, it follows that 
	\begin{align*}
		((\v \times \u)\x, \y)&=(*(\u \wedge *(\v \wedge \x))+\dfrac{2}{3}(\v, \u)\x, \y)
		\\
		&=(*(\u \wedge *(\v \wedge \x)), \y)+\dfrac{2}{3}(\v, \u)(\x, \y)
		\\
		&=-(\x \wedge \v, \y \wedge \u)+\dfrac{2}{3}(\v, \u)(\x, \y)
		\\
		&=-(\y \wedge \u, \x \wedge \v)+\dfrac{2}{3}(\u, \v)(\x, \y)
	\end{align*}
	Hence we have $({}^t (\u \times \v)\x, \y)=((\v \times \u)\x, 
	\y) $ for every $\x, \y \in C^9$, that is, ${}^t(\u \times 
	\v) = \v 
	\times \u$.
	\vspace{1mm}

	(3) 
	Note first that $ *(\tau \u)=\tau(*\u) $. 
	Indeed, for every $ \v \in \varLambda^3(C^9) $, we have 
	\begin{align*}
		(*(\tau \u ), \v )&=(\tau \u \wedge \v, \e_1 \wedge \cdots 
		\wedge \e_9) 
		\\
		&=(\tau (\u \wedge \tau \v), \tau(\e_1 \wedge \cdots \wedge \e_9))
		\\
		&=\tau(\u \wedge \tau \v, \e_1 \wedge \cdots \wedge \e_9)
		\\
		&=\tau(*\u, \tau \v)=(\tau(*\u), \tau(\tau\v))
		\\
		&=(\tau(*\u), \v).
	\end{align*}
	Then using this relational formula, it follows that
  \begin{align*}
	(\tau \u \times \tau \v)\x &=*(\tau\v \wedge *(\tau\u \wedge 
	\x))+\dfrac{2}{3}(\tau \u, \tau \v)\x, \,\,\x \in C^9
	\\
	&=*(\tau\v \wedge \tau (*(\u \wedge \tau\x)))+\dfrac{2}{3}
	\tau(\u, \v)\x
	\\
	&= \tau(*(\v \wedge *(\u \wedge \tau\x)))+\dfrac{2}{3}\tau(\u, 
	\v)\x
	\\
	&= \tau(*(\v \wedge *(\u \wedge \tau\x))+\dfrac{2}{3}(\u, \v)
	(\tau \x))
	\\
	&= \tau((\u \times \v)(\tau\x))
	\\
	&= \tau(\u \times \v)\x
   \end{align*}
   for every $ \x \in C^9 $.
   Hence we have $\tau(\u \times \v)=\tau \u \times \tau \v$.
	\vspace{1mm}

 (4) First, as in (1) above, we do straightforward computation of 
 $D(\u \times \v)\x, \x \in C^9$. It follows from Lemma 
 \ref{lemma 
 4.1} that
 \begin{align*}
	D(\u \times \v)\x&=D(*(\v \wedge *(\u \wedge \x))+\dfrac{2}{3}
	(\u, \v)\x)
	\\
	&=D(*(\v \wedge *(\u \wedge \x)))+\dfrac{2}{3}(\u, \v)(D \x)
	\\
	&=*(-{}^t D(\v \wedge *(\u \wedge \x)))+\dfrac{2}{3}(\u, \v)(D 
	\x)
	\\
	&=*(-{}^t D\v \wedge *(\u \wedge \x)+\v \wedge *(D\u \wedge \x)+
	\v \wedge *(\u \wedge D\x))+\dfrac{2}{3}(\u, \v)D \x.
 \end{align*}
	Then, as in the proof of Lemma \ref{lemma 4.2}, note that $\tr(D)=0$, we have the following
	\begin{align*}
		\tr\,(D(\u \times \v))
		&=\sum_{j=1}^9 (D(\u \times \v)\e_j, \e_j)
		\\
		&=\sum_{j=1}^9 (*(-{}^t D\v \wedge *(\u \wedge \e_j)+\v \wedge *(D\u \wedge \e_j)+\v \wedge *(\u \wedge D\e_j))
		\\[-3mm]
		&\hspace*{60mm}
		+\dfrac{2}{3}(\u, \v) D \e_j, \e_j)
		\\
		&=\sum_{j=1}^9 (*(-{}^t D\v \wedge *(\u \wedge \e_j)+\v \wedge *(D\u \wedge \e_j)+\v \wedge *(\u \wedge D\e_j)),\e_j)
		\\[-3mm]
		&\hspace*{60mm}+\dfrac{2}{3}(\u, \v)\,\sum_{j=1}^9(D \e_j, \e_j)
		\\
		&=\sum_{j=1}^9 (*(-{}^t D\v \wedge *(\u \wedge \e_j)+\v \wedge *(D\u \wedge \e_j)+\v \wedge *(\u \wedge D\e_j)),\e_j)
		\\[-3mm]
		&\hspace*{60mm}+\dfrac{2}{3}(\u, \v)\tr(D)
		\\
		&=\sum_{j=1}^9 (*(-{}^t D\v \wedge *(\u \wedge \e_j)+\v \wedge *(D\u \wedge \e_j)+\v \wedge *(\u \wedge D\e_j)),\e_j)
		\\
		&=\sum_{j=1}^9\Bigl((*(-{}^t D\v \wedge *(\u \wedge \e_j)),\e_j)+(*(\v \wedge *(D\u \wedge \e_j)),\e_j)
		\\[-3mm]
		&\hspace*{60mm}+(*(\v \wedge *(\u \wedge D\e_j)),\e_j)\Bigr)
		\\
		&=-\sum_{j=1}^9\Bigl((\e_j \wedge \u, \e_j \wedge -{}^t D\v)+(\e_j \wedge D\u, \e_j \wedge \v)+(D\e_j \wedge \u, \e_j \wedge \v) \Bigr).
	\end{align*}
    Here, as for computation above, we do the following computation:
    \begin{align*}
    	(\e_j \wedge \u, \e_j \wedge -{}^t D\v)&=(\e_j \wedge \u, -{}^tD(\e_j \wedge \v)+{}^tD\e_j \wedge \v)
    	\\
    	&=(\e_j \wedge \u, -{}^tD(\e_j \wedge \v))+(\e_j \wedge \u,{}^tD\e_j \wedge \v)
    	\\
    	&=-(D(\e_j \wedge \u), \e_j \wedge \v)+(\e_j \wedge \u,{}^tD\e_j \wedge \v)
    	\\
    	&=-(D\e_j \wedge \u+\e_j \wedge D\u, \e_j \wedge \v)+(\e_j \wedge \u,{}^tD\e_j \wedge \v)
    	\\
    	&=-(D\e_j \wedge \u,\e_j \wedge \v)-(\e_j \wedge D\u,\e_j \wedge \v)+(\e_j \wedge \u,{}^tD\e_j \wedge \v).
    \end{align*}
    Hence, combining both computations above, we have  
    \begin{align*}
    \tr\,(D(\u \times \v))=-\displaystyle{\sum_{j=1}^9}(\e_j \wedge \u,{}^tD\e_j \wedge \v).
    \end{align*}
     Thus, as in the proof of Lemma \ref{lemma 4.2}, we obtain $\tr\,(D(\u \times \v))=(D\u, \v)$.
\end{proof}

Note that Theorems \ref{lemma 4.1}, \ref{lemma 4.2} and \ref{lemma 
4.3} also hold with respect to an exterior $C$-vector space 
$\varLambda^k(C^n)$ and an $n$-dimensional $C$-vector space $C^n$, and accordingly, the Lie algebra $
\mathfrak{sl}(9,C)$ and the group $SL(9,C)$ are 
replaced by the Lie algebra $\mathfrak{sl}(n,C)$ and the 
group $SL(n,C)$ (see \cite{go} for details).
\vspace{1mm}

Now, we construct another $C$-Lie algebra ${\mathfrak{e}_8}
^{C}$ of 
type $E_8$. Hereafter, we use the same notation ${\mathfrak{e}_8}
^{C}$ used in previous section as $C$-Lie algebra 
of type 
$E_8$.

\begin{thm}[{\cite[pp.599-600]{go}}]\label{theorem 4.4}
In the $248$-dimensional $C$-vector space
\begin{align*}
  {\mathfrak{e}_8}^{C} = \mathfrak{sl}(9, C) 
  \oplus {\varLambda}^3(C^9) \oplus {\varLambda}
  ^3(C
  ^9), 
\end{align*}
for $R_1=(D_1, \u_1, \v_1), R_2=(D_2, \u_2, \v_2) \in {\mathfrak{e}
_8}^C$, we define a Lie bracket $[R_1, R_2]$ by
$$
   [(D_1, \u_1, \v_1), (D_2, \u_2, \v_2)] =: (D, \u, \v), 
$$
 where
$$
\left\{
\begin{array}{l}
   D = [D_1, D_2] + \u_1 \times \v_2 - \u_2 \times \v_1,
\vspace{1mm}\\
   \u = D_1\u_2 - D_2\u_1 + *(\v_1 \wedge \v_2),
\vspace{1mm}\\
   \v = - {}^tD_1\v_2 + {}^tD_2\v_1 - *(\u_1 \wedge \u_2),
\end{array} \right. 
$$
then ${\mathfrak{e}_8}^{C}$ becomes a $ C$-Lie 
algebra. 
\end{thm}
In order to prove the Jacobi identity, we need the following Lemma. 

\begin{lem}[{\cite[Lemma 2.1]{go}}]\label{lemma 4.5}
For $\u, \v, \w \in {\varLambda}^3(C^9)$, we have the 
following
\vspace{1mm}

{\rm(1)} \,\,$ \u \times *(\v \wedge \w) + \v \times *(\w \wedge \u) + \w \times *(\u \wedge \v) = 0$.
\vspace{1mm}

{\rm (2)} \,\,$(\u \times \w)\v - (\v \times \w)\u + *(*(\u \wedge \v) \wedge \w) = 0$.
\end{lem}

We will return the proof of Theorem  \ref{theorem 4.4}.
\begin{proof} 
 We start the proof of theorem. Let $R_i=(D_i,\u_i,\v_i),i=1,2$ 
 and $R_3 \in {\mathfrak{e}_8}^{C}$. First, it is clear 
 that 
 $[R_1 +R_2, R_3]=[R_1, R_3]+[R_2, R_3]$ and $[\lambda R_1, R_2]=
 \lambda [R_1, R_2], \lambda \in C$. Next, for $\u, \v \in 
 {\varLambda}^3(C^9)$, since we confirm $*(\u \times \v)=-
 *(\v \times \u)$, we have $[R_1, R_2]=-[R_2, R_1]$. Finally, we 
 have 
 to prove the Jacob identity. In order to prove this, using Lemmas 
 \ref{lemma 4.1}, \ref{lemma 4.3}, 
 \ref{lemma 4.5} and note that the relational formulas $*(\v_i \wedge *(\u_k \wedge \u_l))=-*(*(\u_k \wedge \u_l) \wedge \v_k), D_i*(\v_k \wedge \v_l)+*(\v_l \wedge (-{}^tD_i \v_k))+*(\v_k \wedge {}^tD_i \v_l)=0$ hold, it follows from 
 \begin{align*}
 &[R_i,[R_k,R_l]]=:(D,\u,\v),
 \\
 &D=[D_i,[D_k,D_l]]+D_i\u_k\times \v_l+\u_1\times (-{}^tD_i\v_l)-D_i\u_l\times \v_k-\u_l\times ({}^tD_i\v_k)
 \\
 &\qquad -\u_i\times {}^tD_k\v_l+\u_i\times {}^tD_l\v_k-\u_i\times (*(\u_k\wedge\u_l))
 \\
 &\qquad -D_k\u_l\times \v_i-(*(\v_k\wedge\v_l))\times\v_i,
 \\
 &\u=D_iD_k\u_l-D_iD_l\u_k+D_i*(\v_k\wedge\v_l)-[D_k,D_l]\u_i-(\u_k\times \v_l)\u_i
 \\
 &\qquad +(\u_l \times\v_k)\u_i+*(\v_i\wedge(-{}^tD_k\v_l))+*(\v_i\wedge{}^tD_l\v_k)-*(\v_i\wedge(\u_k\wedge\u_l)),
 \\
 &\v={}^tD_i{}^tD_k\v_l-{}^tD_i{}^tD_l\v_1+{}^tD_i*(\u_i\wedge\u_l)+{}^t[D_k,D_l]\v_i+{}^t(\u_k\times \v_l)\v_i
 \\
 &\qquad -{}^t(\u_l\times \v_k)\v_i-*(\u_i\wedge D_k\u_l)+*(\u_i\wedge D_l\u_k)-*(\u_i\wedge *(\v_k\wedge \v_l))
\end{align*}
 that the Jacob identity $[R_1,[R_2,R_3]]+[R_2,[R_3,R_1]]+[R_3,
 [R_1,R_2]]=0$.
\end{proof}

Here, we need the following results.



\begin{thm}[{\cite[Theorem 2.2]{go}}]\label{lemma 4.6}
 The $C$-Lie algebra ${\mathfrak{e}_8}^{C} = 
 \mathfrak{sl}
 (9, C) \oplus {\varLambda}^3(C^9) \oplus 
 {\varLambda}^3(C^9)$ is a simple Lie algebra of type 
 $E_8$.
\end{thm}



\begin{prop}[{\cite[p.602]{go}}]\label{propositin 4.7}
The Killing form $B_8$ of the $C$-Lie algebra $
{\mathfrak{e}_8}^{C} = \mathfrak{sl}(9, C) \oplus 
{\varLambda}^3(C^9) \oplus {\varLambda}^3(C^9)$ is 
given by
$$
   B_8((D_1, \u_1, \v_1), (D_2, \u_2, \v_2)) = 60(\tr(D_1D_2) + (\u_1, \v_2) + (\u_2, \v_1)). 
$$
\end{prop}

We define a complex-conjugate linear transformation $\tau
\ti{\lambda}$ of ${\mathfrak{e}_8}^{C}$ by
$$
    \tau\ti{\lambda}(D, \u, \v) = (-\tau{}^tD, -\tau\v, -\tau\u),
$$
and using this linear transformation we define an Hermitian inner 
product $\langle R_1, R_2 \rangle$ in ${\mathfrak{e}_8}^{C}
$ by
$$
    \langle R_1, R_2 \rangle = -B_8(R_1, \tau\ti{\lambda}R_2). 
$$
Then the explicit form of $\langle R_1, R_2 \rangle$ is given as follows:
$$
    \langle R_1, R_2 \rangle = 60(\tr(D_1(\tau{}^tD_2)) + (\u_1, \tau\u_2) + (\v_2, \tau\v_1)). 
$$

Consequently, as mentioned in \cite[p.598]{go}, 
the group 
$$
     E_8 = \{ \alpha \in \mbox{Aut}({\mathfrak{e}_8}^{C}) 
     \, | \, \langle \alpha R_1, \alpha R_2 \rangle = \langle R_1, 
     R_2 
     \rangle \} 
$$
is the connected compact simple Lie group of type $E_8$ and also 
simply connected. 

Then we have the following lemma.

\begin{lem}\label{lemma 4.8}
The Lie algebra	$\mathfrak{e}_8$ of the group $E_8$ is given by
\begin{align*}
\mathfrak{e}_8&= \{R \in {\mathfrak{e}_8}^{C} \,|\,  \tau
\ti{\lambda} R= R \}
\\
&= \{ R=(D, \u, -\tau \u) \in {\mathfrak{e}_8}^{C} \,|\, D 
\in \mathfrak{su}(9), \u \in \varLambda^3(C^9) \}.
\end{align*}
\end{lem}
\begin{proof}
By doing straightforward computation, we can confirm the required result.
\end{proof}

We define a $C$-linear transformation $w_{{}_4}$ of $
{\mathfrak{e}_8}^{C}$ by
\begin{align*}
		w_{{}_4}(D, \u, \v) = (A_4D{A_4}^{-1}, A_4 \u, {}^t {A_4}^{-1}\v), 
\end{align*}
where $A_4=\diag(1,i,\ldots, i) \in SU(9)=\{A \in M(9, C)
\,|\, (\tau {}^t\! A)A=E, \det\,A=1 \}$. Then we see that $w_{{}_4} 
\in E_8$ and $(w_{{}_4})^4 = 1$. Hence $w_{{}_4}$ induces the inner 
automorphism $\tilde{w}_{{}_4}$ of order four on $E_8$: $
\tilde{w}_{{}_4}(\alpha)=w_{{}_4} \alpha {w_{{}_4}}^{-1}, \alpha 
\in 
E_8$.
\vspace{1mm}

Now, we will study the subgroup $(E_8)^{w_{{}_4}}$ of $E_8$:
\begin{align*}
	(E_8)^{w_{{}_4}} = \{\alpha \in E_8 \, | \, w_{{}_4}\alpha = \alpha w_{{}_4} \}. 
\end{align*}

The aim of the rest of this section is to determine the structure of the group $(E_8)^{w_{{}_4}}$. 
Before that, we prove lemma needed later.

\begin{lem}\label{lemma 4.9}
The Lie algebra $(\mathfrak{e}_8)^{w_{{}_4}}$ of the group $(E_8)^{w_{{}_4}}$ is given by
\begin{align*}
	(\mathfrak{e}_8)^{w_{{}_4}}=\{(D,0,0) \in \mathfrak{e}_8 \,|\, D \in \mathfrak{s}(\mathfrak{u}(1) \oplus \mathfrak{u}(8))\}.
\end{align*}
In particular, we have  $\dim\,((\mathfrak{e}_8)^{w_{{}_4}})=(64+1)-1=64$.
\end{lem}
\begin{proof}
By doing straightforward computation, we can obtain the explicit form above of the Lie algebra $(\mathfrak{e}_8)^{w_{{}_4}}$. Indeed, it follows that
\begin{align*}
(\mathfrak{e}_8)^{w_{{}_4}}&= \{(D, \u, -\tau \u) \in \mathfrak{e}_8 \,|\, w_4 (D, \u, -\tau \u)=(D, \u, -\tau \u) \}
\\
&= \{(D, \u, -\tau \u) \in \mathfrak{e}_8 \,|\, A_4D=DA_4, A_4\u=\u, {}^t {A_4}^{-1}(-\tau\u)=-\tau\u \}
\\
&= \{(D, \u, -\tau \u) \in \mathfrak{e}_8 \,|\, A_4D=DA_4, A_4\u=\u \}
\\
&= \{(D, 0, 0) \in \mathfrak{e}_8 \,|\, D \in \mathfrak{s}(\mathfrak{u}(1) \oplus \mathfrak{u}(8))\}.
\end{align*}

It is clear that $\dim\,((\mathfrak{e}_8)^{w_{{}_4}})=(64+1)-1=64$.
\end{proof}

Now, we will determine the structure of the group $(E_8)^{w_{{}_4}}$.

\begin{thm}\label{theorem 4.10}
The group $(E_8)^{w_{{}_4}}$ is isomorphic to the group $(U(1) \times SU(8))/\Z_{24},\vspace{0.5mm}\,\Z_{24} \\
= \{({\omega_{{}_{24}}}^{3k},{\omega_{{}_{24}}}^{-3k}E),\, ({\omega_{{}_{24}}}^{3k+1},{\omega_{{}_{24}}}^{-3(k-5)}E), \,({\omega_{{}_{24}}}^{3k+2},{\omega_{{}_{24}}}^{-3(k-2)}E)  \,|\, k=0,\ldots,7 \},\vspace{0.5mm} \omega_{24} \allowbreak= e^{i2\pi/24}${\rm :} $(E_8)^{w_{{}_4}} \cong (U(1) \times SU(8))/\Z_{24}$.
\end{thm}
\begin{proof}
We define a mapping $\varphi_{{}_{w_{{}_4}}} :S(U(1) \times U(8)) \to (E_8)^{w_{{}_4}}$ by
\begin{align*}
	 \varphi_{{}_{w_{{}_4}}}(A)(D, \u, \v) = (ADA^{-1}, A\u, {}^tA^{-1}\v). 
\end{align*}

We will prove that $\varphi_{{}_{w_{{}_4}}}$ is well-defined. For 
$R_1=(D_1, \u_1, \v_1), R_2=(D_2, \u_2, \v_2) \in {\mathfrak{e}_8}
^{C}$,
we first show $\varphi_{{}_{w_{{}_4}}}(A)[R_1, R_2]=[\varphi_{{}_{w_{{}_4}}}(A)(R_1),\varphi_{{}_{w_{{}_4}}}(A)(R_2)]$. It follows from Lemmas \ref{lemma 4.1} (2), \ref{lemma 4.3} (1) that 
\begin{align*}
\varphi_{{}_{w_{{}_4}}}(A)[R_1, R_2]=\varphi_{{}_{w_{{}_4}}}(A)([(D_1, \u_1, \v_1), (D_2, \u_2, \v_2)])
=:(ADA^{-1}, A\u, {}^t A^{-1}\v),
\end{align*}
where
\begin{align*}
ADA^{-1}&=A( [D_1, D_2] + \u_1 \times \v_2 - \u_2 \times \v_1)A^{-1}
\\
&= A[D_1, D_2])A^{-1} + A(\u_1 \times \v_2))A^{-1} - A(\u_2 \times \v_1))A^{-1}
\\
&= A(D_1D_2-D_2D_1)A^{-1}+A\u_1 \times  {}^t A^{-1}\v_2-
A\u_2 \times  {}^t A^{-1}\v_1
\\
&=[AD_1A^{-1}, AD_2A^{-1}]+A\u_1 \times  {}^t A^{-1}\v_2-
A\u_2 \times  {}^t A^{-1}\v_1,
\\[2mm]
A\u &= A(D_1\u_2 - D_2\u_1 + *(\v_1 \wedge \v_2))
\\
&= AD_1\u_2 - AD_2\u_1 +A( *(\v_1 \wedge \v_2))
\\
&= (AD_1)\u_2 - (AD_2)\u_1 + *({}^t A^{-1}(\v_1 \wedge \v_2))
\\
&= (AD_1)\u_2 - (AD_2)\u_1 + *({}^t A^{-1}\v_1 \wedge {}^tA^{-1}\v_2),
\\[2mm]
{}^tA^{-1}\v &= {}^t A^{-1}(- {}^tD_1\v_2 + {}^tD_2\v_1 - *(\u_1 \wedge \u_2))
\\
&= {}^tA^{-1}(- {}^tD_1\v_2) + {}^t A^{-1}{}^tD_2\v_1 - {}^t A^{-1}(*(\u_1 \wedge \u_2))
\\
&=  ({}^tA^{-1}(- {}^tD_1))\v_2 +({}^t A^{-1}{}^tD_2)\v_1 - *(A\u_1 \wedge A\u_2), 
\end{align*}
and 
\begin{align*}
[\varphi_{{}_{w_{{}_4}}}(A)(R_1),\varphi_{{}_{w_{{}_4}}}(A)(R_2)] &= [\varphi_{{}_{w_{{}_4}}}(A)(D_1, \u_1, \v_1), \varphi_{{}_{w_{{}_4}}}(A)(D_2, \u_2, \v_2)]
\\
&= [(AD_1A^{-1}, A\u_1, {}^t A^{-1}\v_1), (AD_2A^{-1}, A\u_2, {}^t A^{-1}\v_2)]
\\
&=:(D', \u', \v'),
\end{align*}
where \vspace{-3mm}
\begin{align*}
D' &= [AD_1A^{-1}, AD_2A^{-1}]+A\u_1 \times  {}^t A^{-1}\v_2-
A\u_2 \times  {}^t A^{-1}\v_1,
\\[2mm]
\u' &= (AD_1A^{-1})(A\u_1)-( AD_2A^{-1})(A\u_2)+*({}^t A^{-1}\v_1 \wedge {}^t A^{-1}\v_2)
\\
&= (AD_1)\u_2-( AD_2)\u_1+*({}^t A^{-1}\v_1 \wedge {}^t A^{-1}\v_2),
\\[2mm]
\v' &= -{}^t (AD_1A^{-1})({}^t A^{-1}\v_2)+{}^t (AD_2A^{-1})({}^t A^{-1}\v_1)-*(A\u_1 \wedge A\u_2)
\\
&= ({}^t A^{-1}(- {}^tD_1))\v_2 +({}^t A^{-1}{}^tD_2)\v_1 - *(A\u_1 \wedge A\u_2).
\end{align*}
Hence we have  $\varphi_{{}_{w_{{}_4}}}(A)([R_1, R_2])=[\varphi_{{}_{w_{{}_4}}}(A)(R_1), \varphi_{{}_{w_{{}_4}}}(A)(R_2)]$.

Next, we will show $\langle \varphi_{{}_{w_{{}_4}}}(A)(R_1), \varphi_{{}_{w_{{}_4}}}(A)(R_2) \rangle = \langle R_1, R_2 \rangle$. It follows that 
\begin{align*}
\langle \varphi_{{}_{w_{{}_4}}}(A)(R_1), \varphi_{{}_{w_{{}_4}}}(A)(R_2) \rangle &=
\langle (AD_1A^{-1}, A\u_1, {}^t A^{-1}\v_1), (AD_2A^{-1}, A\u_2, {}^t\! A^{-1}\v_2) \rangle
\\
&= 60((\tr\,(AD_1A^{-1})\tau{}^t(AD_2A^{-1}))+( A\u_1, \tau (A\u_2) 
\\
&\hspace*{60mm}+({}^t A^{-1}\v_2, \tau ({}^t A^{-1}\v_1))
\\
&= 60((\tr\,(AD_1A^{-1}(\tau{}^t A^{-1})(\tau{}^t D_2)(\tau {}^t A))
+( A\u_1, (\tau A)(\tau\u_2)) 
\\
&\hspace*{60mm}+({}^t A^{-1}\v_2, (\tau {}^t A^{-1})(\tau \v_1))
\\
&=60(\tr\,((AD_1)((\tau{}^t D_2) (\tau{}^t A)))+( (\tau{}^t AA)\u_1, \tau\u_2)
\\
&\hspace*{60mm}+ ((\tau A^{-1}{}^t A^{-1})\v_2, \tau \v_1))
\\
&=60(\tr\,((\tau{}^t D_2) (\tau{}^t A))(AD_1)+(\u_1, \tau\u_2)
+ (\v_2, \tau \v_1))
\\
&=60(\tr\,((\tau{}^tD_2) D_1)+(\u_1, \tau\u_2)
+ (\v_2, \tau \v_1))
\\
&=60(\tr\,(D_1(\tau{}^t D_2))+(\u_1, \tau\u_2)
+ (\v_2, \tau \v_1))
\\
&=  \langle R_1, R_2 \rangle,
\end{align*}
that is, $\langle \varphi_{{}_{w_{{}_4}}}(A)(R_1), \varphi_{{}_{w_{{}_4}}}(A)(R_2) \rangle= \langle R_1, R_2 \rangle$.
Hence we have $\varphi_{{}_{w_{{}_4}}}(A) \in E_8$.
Moreover, since $AA_4=A_4A$ holds, we have $w_{{}_4}\varphi_{{}_{w_{{}_4}}}(A)=\varphi_{{}_{w_{{}_4}}}(A)w_{{}_4}$, that is, $\varphi_{{}_{w_{{}_4}}}(A) \in (E_8)^{w_{{}_4}}$. 
Thus $\varphi_{{}_{w_{{}_4}}}$ is well-defined. Obviously $
\varphi_{{}_{w_{{}_4}}}$ is a homomorphism. Here we easily obtain $
\Ker\,\varphi_{{}_{w_{{}_4}}}=\{E, \omega E, \omega^2E \}$, where $ 
\omega \in C, \omega^3=1, \omega \ne1$. Indeed, it follows that
\begin{align*}
\Ker\,\varphi_{{}_{w_{{}_4}}}&= \{A \in S(U(1) \times U(8))\,|\, \varphi_{{}_{w_{{}_4}}}(A)=1   \}
\\
&= \biggl\{A \in S(U(1) \times U(8))\,\biggm|\,\begin{array}{l}ADA^{-1}=D, A\u=\u,{}^t A^{-1}\v=\v,
\\[1mm]
 {\text{for all}}\,\, D \in \mathfrak{sl}(9, C),\u, \v \in 
 \varLambda^3(C^9) \end{array}  \biggr\}
 \\
 &= \biggl\{A \in S(U(1) \times U(8))\,\biggm|\,\begin{array}{l}AD=DA, A\u=\u,(\tau A)\v=\v,
\\[1mm]
 {\text{for all}}\,\, D \in \mathfrak{sl}(9, C),  \u, \v 
 \in \varLambda^3(C^9) \end{array}  \biggr\}
 \\
 &= \{E, \omega E, \omega^2 E \} \cong \Z_3.
\end{align*}

Finally, we will prove that $\varphi_{{}_{w_{{}_4}}}$ is surjective. Since $\Ker\,\varphi_{{}_{w_{{}_4}}}$ is discrete and $(E_8)^{{}_{w_{{}_4}}}$ is connected, and together with $\dim((\mathfrak{e}_8)^{w_{{}_4}})=64=(64+1)-1=\dim(\mathfrak{s}(\mathfrak{u}(1) \oplus \mathfrak{u}(8)))$ (Lemma \ref{lemma 4.9}), we see that $\varphi_{{}_{w_{{}_4}}}$ is surjective.
Thus we have the following  isomorphism 
$$
(E_8)^{w_{{}_4}} \cong S(U(1) \times U(8))/\Z_3.
$$
Further, the mapping $f: U(1) \times SU(8) \to S(U(1) \times U(8))$,
$$
     f(b, B)=\begin{pmatrix} b^{-8} & 0        \\
                                0   & b B
             \end{pmatrix},                         
$$
induces the isomorphism $S(U(1) \times U(8)) \cong (U(1) \times SU(8))/\Z_8, \Z_8=\{({\omega_8}^k, {\omega_8}^{-k} E)\,|\, k\allowbreak=0,1, \ldots, 7  \}$, where $\omega_8=e^{i2\pi/8} \in U(1)$. 

Here, we will determine the kernel of the mapping $\varphi_{{}_{w_4,f}}$ as the composition of the mappings $w_4$ and $f$: 
\begin{align*}
	\varphi_{{}_{w_4,f}}:U(1) \times SU(8) \to S(U(1) \times U(8)) \to (E_8)^{w_{{}_4}}.
\end{align*}
Using the result of $\Ker\,f$, by doing straightforward computation we have the following 
\begin{align*}
	\Ker\,\varphi_{{}_{w_4,f}}\!&=\!\{({\omega_{{}_{24}}}^{\!3k},{\omega_{{}_{24}}}^{\!-3k}E), ({\omega_{{}_{24}}}^{\!3k+1},{\omega_{{}_{24}}}^{\!-3(k-5)}E), ({\omega_{{}_{24}}}^{\!3k+2},{\omega_{{}_{24}}}^{\!-3(k-2)}E)|k=0,\ldots,7  \}
	\\
	&\cong \Z_{24}.
\end{align*} 

Therefore, we have the desired isomorphism 
\begin{align*}
(E_8)^{w_{{}_4}} \cong (U(1) \times SU(8))/\Z_{24}.
\end{align*}
\end{proof} 


\section{Case 3. The automorphism $\tilde{\upsilon}_{{}_4}$ of order four and the group $(E_8)^{{}_{\upsilon_{{}_4}}}$}

In this section (also in the next Section 6), again we use the $248$-dimensional vector space 
${\mathfrak{e}_8}^{C}$ used in Case 1 (\cite{miya2}) and 
the connected compact exceptional Lie group of type $E_8$ constructed 
by T. Imai and I. Yokota (\cite{Imai}). 

We define a $C$-linear transformation $\upsilon_{{}_4}$ of 
${\mathfrak{e}_8}^{C}$  by
$$
\upsilon_{{}_4}(\varPhi, P, Q, r, s, t)=(\varPhi, i P, -i Q, r, -s,  -t).
$$
Then we see that $\upsilon_{{}_4} \in E_8$ and $(\upsilon_{{}
_4})^4=1, (\upsilon_{{}_4})^2=\upsilon$, where $\upsilon$ is the $
C$-linear transformation of ${\mathfrak{e}_8}^{C}$ defined in \cite[Definition of Subsection 5.7(p. 174)]{iy0}, and so $\upsilon_{{}_4} $ induces the inner automorphism $\tilde{\upsilon}_{{}_4}$ of order four on $E_8$: $\tilde{\upsilon}_{{}_4}(\alpha)=\upsilon_{{}_4} \alpha {\upsilon_{{}_4}}^{-1}, \alpha \in E_8$.

Now, we will study the subgroup $(E_8)^{{}_{\upsilon_{{}_4}}}$ of $E_8$:
\begin{align*}
(E_8)^{{}_{\upsilon_{{}_4}}} = \{\alpha \in E_8 \, | \, \upsilon_{{}_4}\alpha = \alpha \upsilon_{{}_4} \}. 
\end{align*}

The aim of the rest of this section is to determine the structure of the group $(E_8)^{{}_{\upsilon_{{}_4}}}$. 
Before that, we prove lemma and proposition needed later.

\begin{lem}\label{lemma 5.1}
	The Lie algebra $(\mathfrak{e}_8)^{{}_{\upsilon_{{}_4}}}$ of the group $(E_8)^{{}_{\upsilon_{{}_4}}}$ is given by
	\begin{align*}
	(\mathfrak{e}_8)^{{}_{\upsilon_{{}_4}}}&=\{ R \in  \mathfrak{e}_8 \,|\, \upsilon_{{}_4} R=R \}\\
	&=\{R=(\varPhi, 0,0,r, 0,0)\,|\, \varPhi \in \mathfrak{e}_7, r \in i\R  \}.
	\end{align*}
\end{lem}
\begin{proof}
	By doing straightforward computation, we can prove this lemma.
\end{proof}

Let $U(1)=\{ \theta \in C\,|\, (\tau \theta)\theta=1 \}$ be the 
unitary group. Then the ordinary unitary group $U(1)$ is isomorphic to 
the group
$\left\{ 
\begin{pmatrix}
  \theta & 0 \\ 
  0      & \tau\theta
\end{pmatrix} \,\Bigm|\theta \in U(1) \right\}$ as the subgroup of the 
group $SU(2)$. 
Hereafter, we denote an element 
$\begin{pmatrix}
\theta & 0 \\ 
0      & \tau\theta
\end{pmatrix}$ by $A_{{}_\theta}$: $A_{{}_\theta}=\begin{pmatrix}
\theta & 0 \\ 
0      & \tau\theta
\end{pmatrix}$.

\begin{prop}\label{proposition 5.2}
 The group $(E_8)^{{}_{\upsilon_{{}_4}}}$ contains a subgroup 
 $$
  \phi_{{}_\upsilon}(U(1))=\{\phi_{{}_\upsilon} (A_{{}_\theta}) \in 
  E_8 \,|\, A_{{}_\theta} \in U(1)   \}
 $$
 which is isomorphic to the group $ U(1)=\{ \theta \in C\,|\, 
 (\tau\theta)\theta=1\} $, 
 where $\phi_{{}_\upsilon}$ is same one as $\varphi_{{}_3}$ defined 
 in {\rm \cite[Theorem 5.7.4]{iy0}}, and moreover the explicit form of $
 \phi_{{}_
 \upsilon} (A_{{}_\theta}): {\mathfrak{e}_8}^{C} \to 
 {\mathfrak{e}_8}^{C}$ is given as follows{\rm :}
 $$
 \phi_{{}_\upsilon} (A_{{}_\theta})(\varPhi, P,Q,r,s,t)=(\varPhi, 
 \theta P,(\tau \theta)Q,r, \theta^2 s,(\tau \theta)^2 t). 
 $$
	
 In particular, we have $\upsilon_{{}_4}=\phi_{{}_\upsilon}(A_{i}), 
 i \in U(1)$.
\end{prop}                                                             
\begin{proof}
	For $A_{{}_\theta}=\begin{pmatrix} \theta & 0 \\
	0 & \tau \theta
	\end{pmatrix}={\rm exp}\begin{pmatrix} -i\nu & 0 \\
	0 & i\nu 
	\end{pmatrix} \in U(1) \subset SU(2)$, \vspace{1mm} by Lemma 
	\ref{lemma 5.1} and \cite[Theorem 5.7.4]{iy0} we have $\phi_
	\upsilon(A_{{}_\theta})={\rm exp}({\ad}(0,0,0, i \nu, 0,0)) \in 
	(E_8)^{{}_{\upsilon_{{}_4}}}$, moreover the explicit form of $ 
	\phi_{{}_\upsilon} (A_{{}_\theta}) $ and the result of $ 
	\upsilon_{{}_4}=\phi_{{}_\upsilon}(A_{i}) $ are the direct results 
	from 
	\cite[Theorem 5.7.4]
	{iy0}).
\end{proof}                                                             

Here, let $ E_7 $ be the compact connected exceptional Lie group of type $ E_7 $ (\cite[Definition of Subsection 4.2 (p.108)]{iy0}), then  we confirm that $ \alpha \in E_7$ satisfies 
$ \alpha 1=1 , \alpha 1^-=1^-$ and $ \alpha 1_-=1_- $, 
where $1=(0,0,0,1,0,0),1^-=(0,0,0,0,1,0),1_-=(0,0,0,0,0,1) \in 
{\mathfrak{e}_8}^{C}$ (see \cite[Lemma 5.7.2, Theorem 5.7.3]
{iy0} for details).
\vspace{2mm}

Now, we will determine the structure of the group $(E_8)^{{}_{\upsilon_{{}_4}}}$.

\begin{thm}\label{theorem 5.3}
	The group $(E_8)^{{}_{\upsilon_{{}_4}}}$ is isomorphic to the group $(U(1) \times E_7)/\Z_2, \Z_2=\{(1, 1),\allowbreak (-1, -1)  \}{\rm :}$ $(E_8)^{{}_{\upsilon_{{}_4}}} \cong  (U(1) \times E_7)/\Z_2$.
\end{thm}
\begin{proof}
	We define a mapping $\varphi_{{}_{\upsilon_{{}_4}}}: U(1) \times E_7 
    \to (E_8)^{{}_{\upsilon_{{}_4}}}$ by
    	\begin{align*}
    		\varphi_{{}_{\upsilon_{{}_4}}}(\theta, \delta)=\phi_{{}_
    		\upsilon}(A_{{}_\theta})\delta.
    	\end{align*}

	We will prove that $ \varphi_{{}_{\upsilon_{{}_4}}} $ is well-
	defined. First, we can confirm that the group $E_7$ is the subgroup 
	of the group $(E_8)^{{}_{\upsilon_{{}_4}}}$. Indeed, let $\delta \in 
	E_7$, then it follows that 
	\begin{align*}
	\upsilon_{{}_4} \delta (\varPhi, P,Q,r,s,t)
	&=\upsilon_{{}_4} (\delta \varPhi \delta^{-1},\delta P,\delta Q,r,s,t),\\
	&= (\delta \varPhi \delta^{-1},i\delta P,-i\delta Q,r,-s,-t)\\
	&=(\delta \varPhi \delta^{-1},\delta (iP),\delta (-iQ),r,-s,-t)\\
	&=\delta \upsilon_{{}_4}(\varPhi, P,Q,r,s,t),\,\, (\varPhi, 
	P,Q,r,s,t) \in {\mathfrak{e}_8}^{C},
	\end{align*}
	that is, $\upsilon_{{}_4} \delta=\delta \upsilon_{{}_4}$. 
    Hence we have $E_7 \subset (E_8)^{{}_{\upsilon_{{}_4}}}$, so 
    together with Proposition 
	\ref{proposition 5.2}, we see that $\varphi_{{}_{\upsilon_{{}
	_4}}}$ is well-defined. Subsequently, since the mapping $\varphi_{{}
	_{\upsilon_{{}_4}}}$ is the restriction mapping of $\varphi_{{}_
	\upsilon}: SU(2) \times E_7 \to (E_8)^\upsilon$ (\cite[Theorem 
	5.7.6]
	{iy0}), $\varphi_{{}_{\upsilon_{{}_4}}}$ is a homomorphism. 
	
	Next, we will prove that $\varphi_{{}_{\upsilon_{{}_4}}}$ is surjective. Let $\alpha \in (E_8)^{{}_{\upsilon_{{}_4}}} \subset (E_8)^\upsilon$. Then there exist $A \in SU(2)$ and $\delta \in E_7$ such that $\alpha=\varphi_{{}_\upsilon}(A, \delta)$. Moreover, from the condition $\upsilon_{{}_4} \alpha {\upsilon_{{}_4}}^{-1}=\alpha $, that is, $\upsilon_{{}_4} \varphi_{{}_\upsilon}(A, \delta) {\upsilon_{{}_4}}^{-1}=\varphi_{{}_\upsilon}(A, \delta)$, as  $A=\begin{pmatrix} 
	a & b \\
	c & d
	\end{pmatrix} \in SU(2)$, this relational formula is expressed as follows:
	$$
	\varphi_{{}_\upsilon}(\begin{pmatrix} a & -b \\
	-c & d
	\end{pmatrix}, \upsilon_{{}_4} \delta {\upsilon_{{}_4}}^{-1})=\varphi_{{}_\upsilon}(\begin{pmatrix} a & b \\
	c & d
	\end{pmatrix}, \delta).
	$$
	Indeed, from Proposition \ref{proposition 5.2}, we easily see  
	$$
	\upsilon_{{}_4} \phi_{{}_\upsilon} (A) {\upsilon_{{}_4}}^{-1}=\phi_{{}_\upsilon}( \begin{pmatrix} i & 0 \\
	0 & -i
	\end{pmatrix}) \phi_{{}_\upsilon}(\begin{pmatrix} a & b \\
	c & d
	\end{pmatrix})\phi_{{}_\upsilon}( \begin{pmatrix}- i & 0 \\
	0 & i
	\end{pmatrix})=\phi_{{}_\upsilon}(\begin{pmatrix} a & -b \\
	-c & d
	\end{pmatrix}), 
	$$ 
	and so it follows that
\begin{align*}
	\upsilon_{{}_4} \varphi_{{}_\upsilon}(A, \delta) {\upsilon_{{}_4}}^{-1}&=\upsilon_{{}_4} (\phi_{{}_\upsilon}(A)\delta) {\upsilon_{{}_4}}^{-1}=(\upsilon_{{}_4} \phi_{{}_\upsilon}(A) {\upsilon_{{}_4}}^{-1})(\upsilon_{{}_4} \delta {\upsilon_{{}_4}}^{-1})\\
	&=\phi_{{}_\upsilon}(\begin{pmatrix} 
              a & -b \\
             -c & d
	\end{pmatrix})(\upsilon_{{}_4} \delta {\upsilon_{{}_4}}^{-1})
	\\
	&=\varphi_{{}_\upsilon}(\begin{pmatrix} 
	          a & -b \\
			 -c & d
	\end{pmatrix}, \upsilon_{{}_4} \delta {\upsilon_{{}_4}}^{-1}).
\end{align*}
	Thus we have the following
	$$
	\left\{\begin{array}{l}
	\begin{pmatrix} 
	         a & -b \\
		    -c & d
	\end{pmatrix} = \begin{pmatrix} 
			a & b \\
			c & d
	\end{pmatrix}
	\vspace{2mm}\\
	\upsilon_{{}_4} \delta {\upsilon_{{}_4}}^{-1} = \delta
	\end{array} \right.         
	\qquad   \text{or}\qquad
	\left\{\begin{array}{l}
	\begin{pmatrix} 
			a & -b \\
			-c & d
	\end{pmatrix} 
	= -\begin{pmatrix} 
			a & b \\
			c & d
	\end{pmatrix}
	\vspace{2mm}\\
	\upsilon_{{}_4} \delta {\upsilon_{{}_4}}^{-1} = -\delta.
	\end{array} \right.   
	$$ 
	In the former case, from the first condition we have $ A=\diag(a,
	\tau a) ,
a \in U(1)$. Since $\delta \in E_7$ leaves the 
condition $\upsilon_{{}_4} \delta {\upsilon_{{}_4}}^{-1} = \delta$ 
as mentioned in the beginning of this proof, it is trivial $\delta 
\in E_7$. Hence there exist $ \theta \in U(1) $ and $ \delta \in E_7 
$ such that $ \alpha=\varphi_{{}_\upsilon}(A_{{}_\theta},\delta)=
\varphi_{{}_{\upsilon_4}}(\theta,\delta) $. 
In the latter case, since we have $
\delta=-\delta$ from the second condition, that is, $\delta=0$, this case is impossible. The proof 
of surjective is completed. 

Finally, from ${\rm Ker}\,\varphi_{{}_\upsilon}\!=\!\{ (E,1), (-E, -1) 
\}$, we easily obtain $\Ker\,\varphi_{{}_{\upsilon_4}}\!\!=\!
\{ (E,1), (-E, -1) \} \allowbreak \cong \Z_2$.

Therefore, we have the desired isomorphism 
\begin{align*}
(E_8)^{{}_{\upsilon_{{}_4}}} \cong  (U(1) \times E_7)/\Z_2
\end{align*}
\end{proof} 

\section{Case 4. The automorphism $\tilde{\mu}_{{}_4}$ of order four and the group $(E_8)^{{}_{\mu_{{}_4}}}$}


We define $ C $-linear transformations of $ \mathfrak{P}^C $ by
\begin{align*}
\lambda(X,Y,\xi,\eta)=(Y,-X,\eta, -\xi), \quad
\gamma(X,Y,\xi,\eta)=(\gamma X,\gamma Y,\xi,\eta),
\end{align*}
where $ \gamma $ on the right hand side is same one as $ \gamma \in G_2 
\subset F_4 \subset E_6 $ (see \cite[Subsections 1.10(p.19), 2.11(p.59), 3.11(p.90)]{iy0} for details) and $ \mathfrak{P}^C $ is the Freudenthal $ C $-vector space (see \cite[Section 2 (p.94)]{iy0}). Then we have $ \lambda \in E_7, \lambda^4=1, \lambda^2=-1 $ (\cite[Lemma 4.3.3]{iy0}) and $ \gamma \in E_7, \gamma^2=1 $.  

Using these transformations, we define  a $C$-linear 
transformation $\mu_{{}_4}$ of $
{\mathfrak{e}_8}^{C}$ by
$$
\mu_{{}_4}(\varPhi, P, Q, r, s, t)=(\lambda\gamma\varPhi \gamma
\lambda^{-1}, -\lambda\gamma P, -\lambda\gamma Q, r, s,  t). 
$$
Then we see $\mu_{{}_4} \in E_8$ and $(\mu_{{}_4})^4=1, (\mu_{{}_4})^2=\upsilon$, where $\upsilon$ is same one in previous section, and so ${\mu_{{}_4}} $ induces the inner automorphism $\tilde{\mu}_{{}_4}$ of order four on $E_8$: $\tilde{\mu}_{{}_4}(\alpha)=\mu_{{}_4} \alpha {\mu_{{}_4}}^{-1}, \alpha \in E_8$.

Now, we will study the subgroup $(E_8)^{{}_{\mu_{{}_4}}}$ of $E_8$:
\begin{align*}
(E_8)^{{}_{\mu_{{}_4}}} = \{\alpha \in E_8 \, | \, \mu_{{}_4}\alpha = \alpha\mu_{{}_4} \}. 
\end{align*}

The aim of the rest of this section is to determine the structure of the group $(E_8)^{{}_{\mu_{{}_4}}}$. 
Before that, we prove lemma and propositions needed later.

\begin{lem}\label{lemma 6.1}
	The Lie algebra $(\mathfrak{e}_8)^{{}_{\mu_{{}_4}}}$ of the group $(E_8)^{{}_{\mu_{{}_4}}}$ is given by
	\begin{align*}
	(\mathfrak{e}_8)^{{}_{\mu_{{}_4}}}&=\{ R \in  \mathfrak{e}_8 \,|\, \mu_{{}_4} R=R \}\\
	&=\{R=(\varPhi, 0,0,r, s,-\tau s)\,|\, \varPhi \in (\mathfrak{e}
	_7)^{\lambda\gamma}, r \in i\R, s \in C \}.
	\end{align*}
\end{lem}
\begin{proof}
	By doing straightforward computation, we can prove this lemma.
\end{proof}

\begin{prop}\label{proposition 6.2}
 The group $(E_8)^{{}_{\mu_{{}_4}}}$ contains a subgroup
 $$
 \phi_{{}_\upsilon}(SU(2)) = \{ \phi_{{}_\upsilon}(A) \in E_8 \, | 
 \, A \in SU(2) \} 
 $$
 which is isomorphic to the group $SU(2) = \{ A \in M(2, C) 
 \, | \, (\tau\,{}^t\!A)A = E, \det A = 1 \}$, where $\phi_{{}_
 \upsilon}$ is defined in Proposition {\rm \ref{proposition 5.2}}.
	
\end{prop}
\begin{proof}
	As in the proof of Proposition \ref{proposition 5.2}, 
	for \vspace{1mm}$A = \begin{pmatrix}
	a & -\tau b \\
	b & \tau a
	\end{pmatrix}=
	{\rm exp}\begin{pmatrix}
	- i\nu & - \tau\rho \\
	\rho & i\nu
	\end{pmatrix} \in SU(2)$, by Lemma \ref{lemma 6.1} and \cite[Theorem 
	5.7.4]{iy1} we have $\phi_{{}_\upsilon}(A) = {\rm 
	\exp}({\rm ad}(0, 0, 0, i\nu, $ $\rho, -\tau\rho)) \in (E_8)^{{}
	_{\mu_{{}_4}}}$.
\end{proof}

Here, again let $ E_7 $ be the compact connected exceptional Lie group of type $ E_7 $, we consider the following subgroup $ (E_7)^{\lambda\gamma} $ of $ E_7 $:
\begin{align*}
 (E_7)^{\lambda\gamma}=\{\alpha \in E_7 \,|\,(\lambda\gamma)\alpha(\gamma\lambda^{-1})=\alpha  \}.
\end{align*}

Then we have the following proposition.

\begin{prop}\label{proposition 6.3}
	The group $(E_7)^{\lambda\gamma}$ is the subgroup of the group $(E_8)^{{}_{\mu_{{}_4}}}${\rm:}$(E_7)^{\lambda\gamma} \subset (E_8)^{{}_{\mu_{{}_4}}}$.
\end{prop}
\begin{proof}
	Let $\alpha \in (E_7)^{\lambda\gamma}$.  
	Note that $-1 \in z(E_7)$ (the center of $E_7$),  we have the following 
	\begin{align*}
	\mu_{{}_4} \alpha(\varPhi, P, Q, r, s, t)&=\mu_{{}_4} (\alpha\varPhi\alpha^{-1}, \alpha P, \alpha Q, r, s, t)\\
	&=(\lambda\gamma \alpha \varPhi \alpha^{-1} \gamma\lambda^{-1}, -\lambda\gamma \alpha P, -\lambda\gamma \alpha Q, r, s, t)\\
	&=(\alpha(\lambda\gamma \varPhi \gamma\lambda^{-1})\alpha^{-1}, \alpha(-\lambda\gamma  P), \alpha(-\lambda\gamma  Q), r, s, t)\\
	&=\alpha\mu_{{}_4} (\varPhi, P, Q, r, s, t),\,\,(\varPhi, P, Q, r, 
	s, t) \in {\mathfrak{e}_8}^{C}, 
	\end{align*}
	that is, $\mu_{{}_4} \alpha=\alpha\mu_{{}_4} $.
	Hence we have the required result $(E_7)^{\lambda\gamma} \subset 
	(E_8)^{{}_{\mu_{{}_4}}}$.
\end{proof}
\vspace{1mm}

Now, we will determine the structure of the group $(E_8)^{{}_{\mu_{{}_4}}}$.

\begin{thm}\label{theorem 6.4}
	The group $(E_8)^{{}_{\mu_{{}_4}}}$ is isomorphic to the group $(SU(2) \times SU(8))/\Z_4, \Z_4=\{(E,E), (E, \allowbreak -E), (-E, e_1E), (-E, -e_1E)\}${\rm :} $(E_8)^{{}_{\mu_{{}_4}}} \cong  (SU(2) \times SU(8))/\Z_4$.
\end{thm}
\begin{proof}
	We define a mapping $\varphi_{{}_{\mu_{{}_4}}}: SU(2) \times SU(8) \to (E_8)^{{}_{\mu_{{}_4}}}$ by
	$$
	\varphi_{{}_{\mu_{{}_4}}}(A, L)=\phi_{{}_\upsilon}(A)\varphi_{{}_{SU(8)}}(L),
	$$
	where the mapping $\varphi_{{}_{SU(8)}}: SU(8) \to (E_7)^{\tau\gamma}=(E_7)^{\lambda\gamma} \subset E_7$ is same one as $\varphi$ defined in \cite[Theorem 4.12.5]{iy0}. From Propositions \ref{proposition 6.2}, \ref{proposition 6.3}, we see that $\varphi_{{}_{\mu_{{}_4}}}$ is well-defined. 
	Subsequently, as in the proof of Theorem \ref{theorem 5.3}, since the mapping $\varphi_{{}_{\mu_{{}_4}}}$ is the restriction mapping of $\varphi_{{}_\upsilon}: SU(2) \times E_7 \to (E_8)^\upsilon$ (\cite[Theorem 5.7.6]{iy0}), $\varphi_{{}_{\mu_{{}_4}}}$ is a homomorphism.
	 
	Next, we will prove that $\varphi_{{}_{\mu_{{}_4}}}$ is surjective. Let $\alpha \in (E_8)^{{}_{\mu_{{}_4}}} \subset (E_8)^\upsilon$. Then there exist $A \in SU(2)$ and $\delta \in E_7$ such that $\alpha=\varphi_{{}_\upsilon}(A, \delta)$. Moreover, from the condition ${\mu_{{}_4}} \alpha {{\mu_{{}_4}}}^{-1}=\alpha $, that is, ${\mu_{{}_4}} \varphi_{{}_\upsilon}(A, \delta) {{\mu_{{}_4}}}^{-1}=\varphi_{{}_\upsilon}(A, \delta)$, this relational formula is expressed by
\begin{align*}
			\varphi_{{}_\upsilon}(A, (\lambda\gamma) \delta (\gamma{\lambda}^{-1}))=\varphi_{{}_\upsilon}(A, \delta).
\end{align*}	
    Indeed, from $\mu_{{}_4}=\mu\lambda\gamma$ and \cite[Lemma 5.4.3]
    {iy1} we easily see $\mu_{{}_4} \phi_{{}_\upsilon}(A) {\mu_{{}_4}}
    ^{-1}=
    \phi_{{}_\upsilon} (A)$,
	and it is clear that $\mu_{{}_4} \delta {\mu_{{}_4}}^{-1}=(\lambda\gamma)\delta(\gamma\lambda^{-1})$ because of $\delta \in E_7$.
	Hence it follows from
	\begin{align*}
	\mu_{{}_4} \varphi_{{}_\upsilon}(A,\delta) {\mu_{{}_4}}^{-1}
	&=\mu_{{}_4} (\phi_{{}_\upsilon}(A) \delta ){\mu_{{}_4}}^{-1}
	=(\mu_{{}_4} \phi_{{}_\upsilon}(A){\mu_{{}_4}}^{-1})(\mu_{{}_4} \delta {\mu_{{}_4}}^{-1})
	\\
	&=\phi_{{}_\upsilon}(A)((\lambda\gamma)\delta(\gamma\lambda^{-1}))
	\\
	&=\varphi_{{}_\upsilon}(A,(\lambda\gamma)\delta(\gamma\lambda^{-1}) )
	\end{align*}
	that $\varphi_{{}_\upsilon}(A, (\lambda\gamma) \delta (\gamma{\lambda}^{-1}))=\varphi_{{}_\upsilon}(A, \delta)$.
	
	\noindent Thus we have the following
	$$
	\left\{\begin{array}{l}
	A = A
	\vspace{1mm}\\
	(\lambda\gamma)\delta (\gamma{\lambda}^{-1}) = \delta
	\end{array} \right.         
	\qquad   \text{or}\qquad
	\left\{\begin{array}{l}
	A= -A
	\vspace{1mm}\\
	(\lambda\gamma)\delta (\gamma{\lambda}^{-1}) = -\delta.
	\end{array} \right.   
	$$
	In the latter case, this case is impossible because of $A\not=0$.
	In the former case, it is clear $ A \in SU(2) $, as for the 
second condition, there exists $L \in 
SU(8)$ such that $\delta=\varphi_{{}_{SU(8)}}(L)$ from \cite[Theorem 4.12.5]{iy0}. 
Hence there exist $ A \in SU(2) $ and $ L \in SU(8) 
$ such that $ \alpha=\varphi_{{}_\upsilon}(A,\varphi_{{}_{SU(8)}}(L))=
\varphi_{{}_{\mu_4}}(A,L) $.
The proof of surjective is completed.
	
Finally, it is not difficult to obtain ${\rm Ker}\,\varphi_{{}_{\mu_{{}
_4}}}=\{(E,E), (E,-E), (-E, e_1E), (-E, -e_1E) \} \allowbreak \cong 
\Z_4$, where 
$e_1$ is one of basis in $\mathfrak{C}$. Indeed, first it follows from 
the definition of kernel that
\begin{align*}
\Ker\,\varphi_{{}_{\mu_{{}_4}}}&=\{(A, L) \in SU(2) \times SU(8)\,|\, 
\varphi_{{}_{\mu_{{}_4}}}(A,L)=1  \}\\
&=\{(A, L) \in SU(2)  \times SU(8)\,|\, \phi_\upsilon(A)\varphi_{{}
_{SU(8)}}(L)=1 \}.
\end{align*}
Here, since the mapping $\varphi_{{}_{\mu_{{}_4}}}$ is the restriction 
of the mapping $\varphi_{{}_\upsilon}:SU(2) \times E_7 \to (E_8)^
\upsilon$, 
it follows from ${\rm Ker}\,\varphi_{{}_\upsilon}=\{(E,1), (-E, -1) \}$ 
(\cite[Theorem 
5.7.6]{iy0}) that 
\begin{align*}
\Ker\,\varphi_{{}_{\mu_{{}_4}}}&=\{(A, L) \in SU(2) \times SU(8)\,|\, 
A=E, \varphi_{{}_{SU(8)}}(L) = 1\}
\\
&\cup \{(A, L) \in SU(2) \times SU(8)\,|\, A=-E, \varphi_{{}_{SU(8)}}(L) 
= -1\}.
\end{align*}	
In the former case, from $\Ker\,\varphi_{{}_{SU(8)}}=\{E,-E \}$, we have
 the following  
	$$
	\left\{\begin{array}{l}
	A = E
	\vspace{1mm}\\
	L=E
	\end{array} \right.         
	\qquad   \text{or}\qquad
	\left\{\begin{array}{l}
	A= E
	\vspace{1mm}\\
	L= -E .
	\end{array} \right.   
	$$
	In the latter case, using $-1=\varphi_{{}_{SU(8)}}(e_1E)=\varphi_{{}
	_{SU(8)}}(-e_1E)$, we have the following
	$$
	\left\{\begin{array}{l}
	A = -E
	\vspace{1mm}\\
	L=e_1E
	\end{array} \right.         
	\qquad   \text{or}\qquad
	\left\{\begin{array}{l}
	A= -E
	\vspace{1mm}\\
	L= -e_1E .
	\end{array} \right.   
	$$
	Hence we have the required result
	\begin{align*}
	\Ker\,\varphi_{{}_{\mu_{{}_4}}}&=\{(E,E), (E,-E), (-E, e_1E), (-E, -e_1E)  \} \cong \Z_4.
	\end{align*}
	
	Therefore, we have the desired isomorphism
	\begin{align*}
	(E_8)^{{}_{\mu_{{}_4}}} \cong  (SU(2) \times SU(8))/\Z_4.
	\end{align*}
\end{proof}

\end{document}